\documentclass{aptpub}
\usepackage[shortlabels]{enumitem}

\authornames{Dan Han et al} 
\shorttitle{} 

\begin{document}

\title{One Explicitly Solvable Model For The Galton-Watson Processes In the Random Environment} 

\authorone[University of Louisville]{Dan Han} 
\authortwo[University of North Carolina at Charlotte; National Research University “Higher School of Economics”]{Stanislav A. Molchanov}
\authorthree[University of North Carolina at Charlotte]{Yanjmaa Jutmaan}

\emailone{dan.han@louisville.edu} 
\addressone{Department of Mathematics, University of Louisville, KY 40292, USA.} 
\addresstwo{Department of Mathematics and Statistics, University of North Carolina at Charlotte, NC 28223, USA; National Research University “Higher School of Economics”, Moscow 101000, Russia.} 
\addressthree{Department of Mathematics and Statistics, University of North Carolina at Charlotte, NC 28223, USA.} 

\begin{abstract}
In this paper, we study the Galton-Watson process in the random environment for the particular case when the number of the offsprings in each generation has the fractional linear generation function with random parameters. In this case, the distribution of $N_t$, the number of particles at the moment time $t=0,1,2,\cdots$ can be calculated explicitly. We present the classification of such processes and limit theorems of two types: quenched type which is for the fixed realization of the random environment and annealed type which includes the averaging over the environment
\end{abstract}



\section{Introduction}
Galton-Watson process in the random environments was intensively discussed in modern probability field (\cite{athreya1971branching2}\cite{athreya1971branching}\cite{vatutin2020initial}\cite{vatutin2020survival}). 
In this paper, we will study one particular model with single-type particles, the discrete time $t=0,1,2,\cdots$ and the fractional linear generating function for the offspring distribution. The fractional linear probability generating function for branching process in the random environment first time was discussed probably in 1994 by D.R. Grey and Lu Zhunwen(\cite{grey1994fractional}),but they focused on the super-critical Smith-Wilkinson branching process and studied the case of 2D environments especially. In 2006, A. Joffe and G. Letac (\cite{joffe2006multitype}) studied the multitype branching process with linear fractional generating functions and the same denominator. In 2019, V.A. Vatutin and his group (\cite{vatutin2019}) investigated the limit behavior of only super-critical multitype branching processes in random environments with linear fractional offspring distributions but this study didn't contain the detailed categorization for such type branching processes. 

Let's give the description of our model. The probability space $(\Omega,\mathcal{F},P)$ in our case is the skew product of two spaces $\Omega=\Omega_m\times \Omega_{br}(\boldsymbol{\omega_m})$. $\Omega_m$ is the space of medium or environment. $\Omega_{br}(\boldsymbol{\omega_m})$ is the space containing the information of the branching and annihilating particles for fixed environment $\boldsymbol{\omega_m}$. Let $N_t$ be the number of the particles at the moment $t$, $t=0,1,2,\cdots$. $N_0=1$ and transition from $N_{t-1}$ to $N_t$ is given by the standard formula:

\begin{equation}
N_t=\sum\limits_{i=1}^{N_{t-1}}\xi_{t,i}(\boldsymbol{\omega_m})
\end{equation}

Here $\xi_{t,i}(\boldsymbol{\omega_m})$ are independent copies of the random variable $\xi_t(\boldsymbol{\omega_m})$, which gives the number of offsprings generated by different particles (among $N_{t-1}$) at the moment $t$. That is
$$P(\xi_t=k)=p_k(\omega_t).\,\,k=0,1,2,\cdots,\,\,\sum\limits_{k=0}^{\infty}p_k=1$$

The distribution of $\xi_{t,i}$ depends on some parameters $\vec{\pi}_t$ selected randomly and independently at the moments $t=1,2,\cdots$. One can identify $\boldsymbol{\omega_m}$ with the sequence $\{\vec{\pi}_t,t=1,2,3,\cdots\}$.

The elementary events associated with the non-homogeneous Galton-Watson process $N_t$ for fixed $\boldsymbol{\omega_m}$, we denote as $\omega_{br}(\boldsymbol{\omega_m})$. They form the second component $\Omega_{br}(\boldsymbol{\omega_m})$ in our skew product $\Omega=\Omega_m \times \Omega_{br}(\boldsymbol{\omega_m})$. 

For each random variable $\eta(\boldsymbol{\omega_m},\omega_{br})$ on the total probability space $\Omega$, one can introduce two types expectations and probabilities:

\begin{enumerate}[1)]
\item Quenched type

We call the following integral the quenched expectation:

\begin{equation}
E\eta=\int_{\Omega_{br}(\boldsymbol{\omega_m})}\eta(\boldsymbol{\omega_m},\omega_{br})d P(d\boldsymbol{\omega_m})
\end{equation}

Here $P^{\boldsymbol{\omega_m}}$ is the conditional distribution of the process $N_t,t\geq 0$ for the fixed $\boldsymbol{\omega_m}$. Similarly, say

\begin{equation}
P(\eta<x)=E\mathbf{1}_{\eta<x}
\end{equation}

the quenched probability. The quenched expectations or probabilities are the random variables on $\Omega_m$.

\item Annealed type

Additional integration over $\Omega_m$ leads to the total expectation and probability, we will use in this case notations $\mathcal{E}$ and $\mathcal{P}$ for expectation and probability. 

\begin{equation}
\mathcal{E}\eta(\boldsymbol{\omega_m},\omega_{br})=\int_{\Omega_m}E\eta\mathcal{P}(d \boldsymbol{\omega_m}) 
\end{equation}
\end{enumerate}

Under the fixed environment realization $\boldsymbol{\omega_m} \in \Omega_m$, the probability generating function of the offspring distribution is $\varphi_t(z,\boldsymbol{\omega_m})=Ez^{\xi_t(\boldsymbol{\omega_m})}$ and the probability generating function of $N_t$ is defined as 
$\Phi_t(z,\boldsymbol{\omega_m})=Ez^{N_t}$, then $\Phi_t(z,\boldsymbol{\omega_m})=\underbrace{\varphi_1(\varphi_2\cdots\varphi_t(z,\boldsymbol{\omega_m}))}_{t-fold}$. We will study the very particular case when $\varphi_t(z,\boldsymbol{\omega_m})$ is the fractional linear function:

\begin{equation}
\varphi_t(z,\boldsymbol{\omega_m})=\frac{a_t z+b_t}{c_t z+d_t}
\end{equation}

This fractional linear function is associated with matrix 
\begin{equation}\label{general mobius matrix}
\mathbf{I}_t=
\begin{pmatrix}
a_t & b_t \\
c_t & d_t
\end{pmatrix}
\end{equation}

We will assume that $\det \mathbf{I}_t>0$. Such fractional linear functions (\text{M\"obius} transforms) form a group with operation of composition and unity $\varphi(z)=z$.
The mapping $\varphi(z)=\displaystyle\frac{az+b}{cz+d}\rightarrow \mathbf{I}=\begin{pmatrix}
a & b \\
c & d
\end{pmatrix}$ is not one-to-one since $\mathbf{I}$ and $c\mathbf{I}$ present the same function $\varphi(z)$, but if we impose the condition $\det \mathbf{I}=1$, then the group of the 
\text{M\"{o}bius} transforms will be isomorphic to the group $SL(2,R)$, this group contains all $2\times 2$ real matrices with determinant 1. But we will not use this fact.

As easy to prove the function $\varphi(z)=\displaystyle\frac{az+b}{cz+d}$ is the generating function of the random variable $\xi$ with values $0,1,2,\cdots n$. if and only if it can be presented in the form 

\begin{equation}
\varphi_t(z)=\frac{(p-\alpha)z+\alpha}{(p-1)z+1}\,\,t=1,2,3,\cdots
\end{equation}

whose associated matrix is 

\begin{equation}
\mathbf{I}=
\begin{pmatrix}
p-\alpha & \alpha \\
p-1 & 1
\end{pmatrix}\end{equation}

where $(p,\alpha)\in [0,1]^2$. We will consider only non-degenerated case $(p,\alpha)\in (0,1)^2$. In this case, $\xi$ has the generalized geometric law:

\begin{equation}
P(\xi=0)=\alpha, P(\xi=k)=(1-\alpha)p(1-p)^{k-1}, k\geq 1
\end{equation}

and
\begin{equation}
\varphi_t(z,\boldsymbol{\omega_m})=\frac{(p_t-\alpha_t) z+\alpha_t}{(p_t-1) z+1}
\end{equation} 

with the associated independent matrices
\begin{equation}
\mathbf{I}=
\begin{pmatrix}
p_t-\alpha_t & \alpha_t \\
p_t-1 & 1
\end{pmatrix},\,\,
\det \mathbf{I}=p_t(1-\alpha_t)
\end{equation}

One can identify $\boldsymbol{\omega_m}$ with sequence of the independent, identically distributed vectors $(p_t,\alpha_t,t\geq 1)=\boldsymbol{\omega_m}$. 
In the future, we will denote $\beta_t=1-\alpha_t$, $q_t=1-p_t$. 

Section 2 in this paper will contain the calculation of the generating function $\Phi_t(z,\boldsymbol{\omega_m})=Ez^{N_t}$ and related objects such as quenched moments etc. In the section 3, we will give the classification of Galton-Watson processes in the random environment. It contains five categories, the standard super-critical and the critical cases of the classical homogeneous Galton-Watson theory will be divided into two sub-categories each. In the section 4, we will prove several quenched and annealed limit theorems using the asymptotic formulas for the random geometric progressions. We discuss the relation of our topic with the problem of the evaluation of the mass of cells in the one cell population (like plankton).

\section{Analysis of the fractional linear model}
Let us calculate the generating function $\Phi_t(z,\boldsymbol{\omega_m})$.

In our case,the matrix $\mathbf{I}_t(\boldsymbol{\omega_m})$ associated with the generating function $\varphi_t(z,\boldsymbol{\omega_m})=\displaystyle\frac{(p_t-\alpha_t)z+\alpha_t}{(p_t-1)z+1}$:

\begin{equation}
\mathbf{I}_t(\omega)=
\begin{pmatrix}
p_t-\alpha_t & \alpha_t \\
p_t-1 & 1
\end{pmatrix}
\end{equation}

The fractional-linear \text(M\"{o}bius) transform functions form a group and 

\begin{equation}
\Phi_t(z,\boldsymbol{\omega_m})=Ez^{N_t(\boldsymbol{\omega}_m)}=\underbrace{\varphi_1(\varphi_2\cdots\varphi_t(z,\boldsymbol{\omega}_m))}_{t-fold}=\frac{A_tz+B_t}{C_tz+D_t}
\end{equation}

and the associated matrix for $\Phi_t(z,\boldsymbol{\omega_m})$ is
$\Pi_t(\boldsymbol{\omega}_m)=
\begin{pmatrix}
A_t & B_t \\
C_t & D_t
\end{pmatrix}
=\mathbf{I}_1\mathbf{I}_2\cdots \mathbf{I}_t(\boldsymbol{\omega}_m)
$

Note the determinant of $I_t$ is $det(\mathbf{I}_t)=p_t(1-\alpha_t)$, thus the determinant of $\Pi_t(\boldsymbol{\omega}_m)=\displaystyle\prod\limits_{i=1}^{t}p_i\prod\limits_{i-1}^t \alpha_i$. And $E[\xi_t(\boldsymbol{\omega}_m)]=\varphi'_t(1)=\displaystyle\frac{1-\alpha_t}{p_t}$, from $N_t=\sum\limits_{i=1}^{N_{t-1}}\xi_{t,i}(\boldsymbol{\omega_m})
$, we can get $E[N_t]=\prod\limits_{i=1}^tE[\xi_{t,i}(\boldsymbol{\omega_m})]=\displaystyle\prod\limits_{i=1}^t \frac{1-\alpha_i}{p_i}=e^{\sum\limits_{s=0}^t\ln \frac{\beta}{p}(s,\boldsymbol{\omega}_m)}$ for the fixed environment realization $\boldsymbol{\omega_m}$. 

In the exponent in the last formula, we have the sum of i.i.d random variables on $\Omega$. Usually, we will assume that $\mathcal{E}\ln\displaystyle\frac{\beta}{p}(\cdot)<\infty$, but some exclusions from this assumption are also possible, we will discuss them later as well as the case $\mathcal{E}\displaystyle\ln^2\frac{\beta}{p}(\cdot)=\infty$. Now we will calculate the product $\Pi_t$,i.e. $\Phi_t(z,\boldsymbol{\omega}_m)=\underbrace{\varphi_1(\varphi_2\cdots\varphi_t(z,\boldsymbol{\omega}_m))}_{t-fold}$.

Let us introduce special $2\times 2$ matrices:

$\mathbf{\Sigma}_1=\begin{pmatrix}
-1 & 1\\
-1 & 1
\end{pmatrix}
$,
$\mathbf{\Sigma}_2=\begin{pmatrix}
1 & 0\\
1 & 0
\end{pmatrix}
$,
$\mathbf{\Sigma}_3=\begin{pmatrix}
1 & -1\\
0 & 0
\end{pmatrix}
$

The following table shows the multiplication results of any two matrices of the above. 

\begin{center}
\begin{tabular}{lllll}
                     & $\mathbf{\Sigma}_1$   & $\mathbf{\Sigma}_2$    & $\mathbf{\Sigma}_3$   \\
 $\mathbf{\Sigma}_1$ &  $\mathbf{O}$        & $\mathbf{O}$           & $\mathbf{\Sigma}_1$  \\
 $\mathbf{\Sigma}_2$ & $\mathbf{\Sigma}_1 $   & $\mathbf{\Sigma}_2$    & $-\mathbf{\Sigma}_1$  \\
 $\mathbf{\Sigma}_3$ &  $\mathbf{O}$         &  $\mathbf{O}$          & $\mathbf{O}$ 
\end{tabular}
\end{center}

where $\mathbf{O}$ is a $2\times2$ matrix with all elements 0.

One can decompose $\mathbf{A}_t$ into the following form:

\begin{equation}
\mathbf{A}_t=\begin{pmatrix}
-1 & 1\\
-1 & 1
\end{pmatrix}+
p_t\begin{pmatrix}
1 & 0\\
1 & 0
\end{pmatrix}+
\beta_t\begin{pmatrix}
1 & -1\\
0 & 0
\end{pmatrix}=
\mathbf{\Sigma}_1+p_t\mathbf{\Sigma}_2+\beta_t\mathbf{\Sigma}_3
\end{equation}

Then $\Pi_t=\mathbf{A}_1\mathbf{A}_2\cdots\mathbf{A}_t=a_t\mathbf{\Sigma}_1+b_t\mathbf{\Sigma}_2+c_t\mathbf{\Sigma}_3$

and
\begin{align*}
\Pi_{t+1}&=a_{t+1}\mathbf{\Sigma}_1+b_{t+1}\mathbf{\Sigma}_2+c_{t+1}\mathbf{\Sigma}_3\\
&=\Pi_t(\mathbf{\Sigma}_1+p_{t+1}\mathbf{\Sigma}_2+\beta_{t+1}\mathbf{\Sigma}_3)\\
&=(a_t\mathbf{\Sigma}_1+b_t\mathbf{\Sigma}_2+c_t\mathbf{\Sigma}_3)(\mathbf{\Sigma}_1+p_{t+1}\mathbf{\Sigma}_2+\beta_{t+1}\mathbf{\Sigma}_3)\\
&=(b_t+a_t\beta_{t+1}-b_t\beta_{t+1})\mathbf{\Sigma}_1+b_tp_{t+1}\mathbf{\Sigma}_2+c_t\beta_{t+1}\mathbf{\Sigma}_3
\end{align*}

It leads to the following iterated equations:

\begin{align}
a_{t+1}&=b_t+a_t\beta_{t+1}-b_t\beta_{t+1}\label{a_t}\\
b_{t+1}&=b_tp_{t+1}\label{b_t}\\
c_{t+1}&=c_t\beta_{t+1}\label{c_t}
\end{align}

From the (\ref{b_t}) and (\ref{c_t}), 

\begin{equation}\label{bt}
b_t=\prod\limits_{i=1}^tp_i
\end{equation}

\begin{equation}\label{ct}
c_t=\prod\limits_{i=1}^t\beta_i
\end{equation}

Substitute (\ref{bt}) and (\ref{ct}) into (\ref{a_t}), we can get 

\begin{equation}\label{a_final}
a_{t+1}=\prod\limits_{i=1}^tp_i+a_t\beta_{t+1}-\prod\limits_{i=1}^tp_i\beta_{t+1}
\end{equation}

Without loss of generality, we assume $a_1=1$ and $\beta_1=1$. 

We can get now the formula for $a_t$:

\begin{align}
a_t&=\prod\limits_{i=1}^{t-1}p_i(1-\beta_t)+\beta_t(1-\beta_{t-1})\prod\limits_{i=1}^{t-2}p_i+\beta_t\beta_{t-1}(1-\beta_{t-2})\prod\limits_{i=1}^{t-3}p_i\nonumber\\
&+\beta_t\beta_{t-1}\beta_{t-2}(1-\beta_{t-3})\prod\limits_{i=1}^{t-4}p_i+\cdots+\beta_t\beta_{t-1}\beta_{t-2}\cdots\beta_2 \label{at}
\end{align}

With (\ref{at}),(\ref{bt}) and (\ref{ct}), 

\begin{equation}
\mathbf{\Pi}_t=\begin{pmatrix}
-a_t+\displaystyle\prod\limits_{i=1}^tp_i+\displaystyle\prod\limits_{i=1}^t\beta_i & a_t-\displaystyle\prod\limits_{i=1}^t\beta_i\\
-a_t+\displaystyle\prod\limits_{i=1}^tp_i & a_t
\end{pmatrix}
=\begin{pmatrix}
A_t & B_t\\
C_t & D_t
\end{pmatrix}
\end{equation}

The determinant of $\mathbf{\Pi}_t$ is $det(\mathbf{\Pi}_t)=\displaystyle\prod\limits_{i=1}^tp_i\prod\limits_{i=1}^t\beta_i$

We know $\mathbf{\Pi}_t$ is the associated matrix for the probability generating function $\Phi_t(z,\omega)=\displaystyle\frac{A_tz+B_t}{C_tz+D_t}$. The probability of the population to extinct at time $t$ is 

\begin{equation}
\pi(t)=P(N_t=0)=\Phi_t(0,\omega)=\frac{B_t}{D_t}=1-\displaystyle\frac{\displaystyle\prod\limits_{i=1}^t\beta_i}{a_t}
\end{equation}

In other words, the survival probability of the population at time $t$ is 

\begin{equation}
P(N_t>0)=1-\pi(t)=\displaystyle\frac{\displaystyle\prod\limits_{i=1}^t\beta_i}{a_t}=\frac{1}{S_t}
\end{equation}

where $S_t=\displaystyle 1+\frac{1-\beta_1}{\beta_1}+\frac{1-\beta_2}{\beta_2}\frac{p_1}{\beta_1}+\cdots \frac{1-\beta_t}{\beta_t}\prod\limits_{j=1}^{t-1}\frac{p}{\beta}(j)$.

Using these notations, we can define another matrix associated with extinction probability:

$\widetilde{\Pi_t}=\displaystyle\frac{1}{a(t)}\Pi_t=\begin{pmatrix}
\rho-\pi & \pi \\
\rho-1   & 1 
\end{pmatrix}(t)
=\begin{pmatrix}
-1+\displaystyle\frac{\displaystyle\prod\limits_{i=1}^t \displaystyle\frac{p_i}{\beta_i}}{S_t}+\displaystyle\frac{1}{S_t} & 1-\displaystyle\frac{1}{S_t}\\
-1+\displaystyle\frac{\displaystyle\prod\limits_{i=1}^t\displaystyle\frac{p_i}{\beta_i}}{S_t}  & 1
\end{pmatrix}$

where $\rho(t)=\displaystyle\frac{\displaystyle\prod\limits_{i=1}^t p_i}{a_t}=\displaystyle\frac{\displaystyle\prod\limits_{i=1}^t \displaystyle\frac{p_i}{\beta_i}}{S_t}$

Thus 
\begin{equation}
\rho(t)=\displaystyle\frac{1}{\displaystyle\sum\limits_{k=1}^{t-1}\displaystyle\frac{1-\beta_{t-k+1}}{\beta_{t-k+1}} \prod\limits_{i=t-k+1}^{t}\displaystyle\frac{\beta_i}{p_i}+\prod\limits_{i=1}^t\frac{\beta_i}{p_i}}
\end{equation}

\begin{theorem}\label{distribution of N_t}
The probability that we have $k$ particles at time $t$ in the fixed random environment $\omega_t$ is $P(N_t=k)=(1-\pi(t))\rho(t)(1-\rho(t))^{k-1}$, $k\geq 1$.

\end{theorem}

\begin{proof}
Recall $\Phi_t(z)$ is the probability generating function of $N_t$ under the fixed random environment.
$\Phi_t(z)=\displaystyle\frac{(\rho-\pi)z+\pi}{(\rho-1)z+1}=\sum\limits_{k=0}^{\infty}P(N_t=k)z^k$. This power series is convergent absolutely for all $|z|<1$.

Thus $\Phi_t(z)-\pi(t)=\displaystyle\frac{\rho(t)(1-\pi(t))z}{(\rho(t)-1)z+1}=\rho(t)(1-\pi(t))z\sum\limits_{n=0}^{\infty}(1-\rho(t))^nz^n$.

Thus $\Phi_t(z)=\pi(t)+\sum\limits_{n=0}^{\infty}(1-\pi(t))\rho(t)(1-\rho)^nz^{n+1}$

In other words, $P(N_t=k)=(1-\pi(t))\rho(t)(1-\rho(t))^{k-1}$, $k\geq 1$.

\end{proof}

\begin{theorem}
Assume
$\rho(t,\boldsymbol{\omega_m})=\displaystyle\frac{\displaystyle\prod\limits_{i=1}^t \displaystyle\frac{p_i}{\beta_i}}{S_t(\boldsymbol{\omega_m})}\xrightarrow[t\rightarrow \infty]{} 0$ $P_m-a.s.$

Then for $a\geq 0$

\begin{equation}
P\left(\frac{N(t)}{E[N_t|N_t\geq 1]}>a|N(t) \geq 1\right)\xrightarrow[t\rightarrow\infty]{} e^{-a}\,\,\, P_m-a.s.
\end{equation}
\end{theorem}
 
\begin{proof}

$P(N_t=k|N_t\geq 1)=\frac{(1-\pi(t))\rho(t)(1-\rho(t))^{k-1}}{1-\pi(t)}=\rho(t)(1-\rho(t))^{k-1}$, $k\geq 1$. This is the pure geometric law. Thus $E[N_t|N_t\geq 1]=\displaystyle\frac{1}{\rho(t)}$ and as $\rho \xrightarrow[t\rightarrow\infty]{} 0$

\begin{align*}
P\left(\frac{N(t)}{E[N_t|N_t\geq 1]}>a|N(t) \geq 1\right)&=P(N_t>\frac{a}{\rho(t)})=(1-\rho(t))^{\frac{a}{\rho(t)}}\rightarrow e^{-a}
\end{align*}

\end{proof}
\section{Classification of Galton-Watson Process in Random Environment}

In the classical homogeneous theory, the classification of Galton-Watson process has three classes:
\begin{enumerate}[a)]

\item super-critical process if $E\xi=a>1$, where $\xi$ is the number of the offsprings $E[N_t]=a^t,a>1$, In other words, $E[N_t] \rightarrow \infty$ exponentially fast as $t\rightarrow \infty$

\item critical process if $E[N_t]=1$ 

\item sub-critical process if $E\xi<1$ (we assume that $E|\xi|<\infty$)

\end{enumerate}
Then see B. Sevastyanov \cite{Sevast}, we have the following classical results:

\begin{enumerate}
\item If $a>1$ then $\displaystyle\frac{N_t}{EN_t}=\frac{N_t}{a^t}\rightarrow N_{\infty}^*$ in law and $P(N_{\infty}^*=0)=\alpha<1$, $\alpha=\lim\limits_{t\rightarrow \infty}P(N_t=0)$
The extinction probability $\alpha$ is the single root of the equation $\alpha=\varphi(\alpha)$, $\alpha<1$.

\item If $E\xi=a=1$ (critical case), then $P(N_t=0)\rightarrow 1$ as $t\rightarrow \infty$ and $P(\displaystyle\frac{N_t}{t}>a|N_t\geq 1)\rightarrow e^{-ca}$ for an appropriate constant $c>0$. 

\item If $E\xi=a<1$ and for $c\in (a,1)$, $EN_t=a^t$,$P(N_t > c^t)\leq \frac{a^t}{c^t}$, $\sum\limits_t P(N_t\geq c^t)<\infty$, that is , population degenerates very fast. 
\end{enumerate}

Similar classfication for the Galton-Watson processes in the random environment is more complicated.

\subsection{Super-critical Galton-Watson process in random environment}
Like in the classical situation, we call the branching process $N_t(\boldsymbol{\omega_m})$ in the random environment supercritical if $\mathcal{E}(\ln \displaystyle\frac{\beta}{p})=\gamma>0$. Then
$E[N_t]=\prod\limits_{i=1}^t\displaystyle\frac{\beta_i}{p_i}=e^{\sum\limits_{i=1}^t\ln( \displaystyle\frac{\beta_i}{p_i})}=e^{t\mathcal{E}(ln  \displaystyle\frac{\beta}{p})+o(t)}$
due to the Strong Law of Large Numbers. That is 

\begin{equation}
\frac{\ln EN_t}{t}\xrightarrow[t\rightarrow \infty]{} \gamma=\mathcal{E}(\ln \displaystyle\frac{\beta}{p})>0\,\, P_m \,\,\text{a.s.}
\end{equation}

The additional classification depends on the $P(N_t=0), t \rightarrow \infty$. In the classical situation, $P(N_t=0)\xrightarrow[t\rightarrow \infty]{} \delta <1$ and $\delta$ is the root of the equation $\varphi(z)=z,z\in [0,1]$. 

In our case, consider now the event $\{N_t\geq 1\}$. Then 

$P(N_t\geq 1)=\displaystyle\frac{1}{S_t}$, $S_t=\displaystyle 1+\frac{1-\beta_1}{\beta_1}+\frac{1-\beta_2}{\beta_2}\frac{p_1}{\beta_1}+\cdots \frac{1-\beta_t}{\beta_t}\prod\limits_{j=1}^{t-1}\frac{p}{\beta}(j)$.

Obviously $P(N_t\geq 1)\rightarrow P(N_\infty\geq 1)=\displaystyle\frac{1}{S_{\infty}}$ as $t\rightarrow \infty$, where $S_{\infty}=\displaystyle 1+\frac{1-\beta_1}{\beta_1}+\frac{1-\beta_2}{\beta_2}\frac{p_1}{\beta_1}+\cdots \frac{1-\beta_t}{\beta_t}\frac{p_{t-1}p_{t-2}\cdots p_1}{\beta_{t-1}\beta_{t-2}\cdots \beta_1}+\cdots$.

Notation $\{N_{\infty}\geq 1\}$ means that the branching process $N_t$ is non-degenerating.

$P(N_{\infty}=0)=1-P(N_{\infty}\geq 1)=1-\displaystyle\frac{1}{S_{\infty}}$. 

When $\gamma=\mathcal{E}\ln \frac{\beta}{p}>0$, the behavior of the series for $S_{\infty}$ depends on the front coefficient $\displaystyle\frac{1-\beta_i}{\beta_i}, i=1,2,3\cdots$. 
\begin{theorem}
Assume that for some small enough $\delta>0$, 
\begin{equation}\label{supercondition}
\sum\limits_{t}\mathcal{P}(\beta_t<exp(-(\gamma-\delta)t)<\infty\,\, \text{and}\,\, \gamma=\mathcal{E}\ln\displaystyle\frac{\beta}{p}>0
\end{equation} 
then the process $N_t\rightarrow \infty$ exponentially fast with uniformly positive probability $P(N_t\geq 1)\geq \delta_1$ for some $\delta_1>0$
\end{theorem}

\begin{proof}
In fact, due to the first Borel Cantelli Lemma, $\beta_t\geq exp(-(\gamma-\delta)t), \forall (\delta<\gamma), t\geq t_o(\boldsymbol{\omega}_m)$. That is $\frac{1-\beta_t}{\beta_t}\leq e^{(\gamma-\delta)t}
, t\geq t_o$ and the series $S_{\infty}$ converges. 

Due to Chebyshev' s inequality condition 
\begin{equation}
\mathcal{E}\ln^{1+\delta}\frac{1}{\beta}(\cdot)\leq c < \infty
\end{equation}

for some $\delta>0$ is sufficient for condition (\ref{supercondition}). 

\end{proof} 

We call such process the \textbf{strong super-critical Galton-Watson process}. 

%
%
%
%
%
%
%

\begin{theorem}
Assume that for some small enough $\delta>0$, 
\begin{equation}\label{weakcon}
\sum\limits_{t}\mathcal{P}(\beta_t<exp(-(\gamma+\delta)t)=\infty\,\, \text{and}\,\, \gamma=\mathcal{E}\ln\displaystyle\frac{\beta}{p}>0
\end{equation} 
then $P$-a.s. the process $N_t\rightarrow \infty$ exponentially fast and $P(N_t=0)\rightarrow 1$ as $t\rightarrow \infty$
\end{theorem}

\begin{proof}
This is the corollary of the second Borel Cantelli Lemma for independent random variable $\beta_t$.
\end{proof}

We call such process the \textbf{weak super-critical Galton-Watson process}. 
Of course, one can prove more precise results considering the central limit theorems for the sums of $\displaystyle\sum\limits_{i=1}^t \ln\frac{\beta}{p}(j)$. 


\subsection{Sub-critical Galton-Watson process in random environment}
We call  Galton-Watson process in random environment sub-critical if $\mathcal{E}(\ln \displaystyle\frac{\beta}{p})=\gamma<0$, then $E[N_t]=\prod\limits_{i=1}^t\displaystyle\frac{\beta_i}{p_i}=e^{\sum\limits_{i=1}^t\ln( \displaystyle\frac{\beta_i}{p_i})}=e^{\gamma t+o(t)}$ and as easy to understand, $S_{\infty}=\infty$ P-a.s.

\begin{theorem}
If $\mathcal{E}(\ln \displaystyle\frac{\beta}{p}) <0$ (sub-critical case), then like in the classical homogeneous situation,  $E[N_t]\rightarrow 0$ as $t\rightarrow \infty$ exponentially fast, and $P(N_t\geq 1)\rightarrow 0$ exponentially fast.
\end{theorem}

The last statement follows from Chebyshev inequality $P(N_t\geq 1)\leq E[N_t]=e^{\gamma t+o(t)}$.

Define a random variable $\tau(\omega_m)=min\{t:N_t=0\}$, this is the extinction moment. Its distribution is given by the formula

\begin{equation}
P(\tau>t)=P(N_t\geq 1)=\frac{1}{S_t}
\end{equation}

Then

\begin{equation}
E[\tau ]=\displaystyle \sum \limits_{t=0}^{\infty}P(\tau>t)=\sum\limits_{t=0}^{\infty}\frac{1}{S_t}
\end{equation}

Since $P(\tau>t)=P(N_t\geq 1)\leq E[N_t]\leq e^{(\gamma+\epsilon)t}$ for $t\geq t_0(\boldsymbol{\omega}_m)$ and any sufficiently small $\epsilon$. As a result, $E[\tau]=\sum\limits_{t=0}^{\infty}\frac{1}{S_t}<\infty$ and even $E[\tau^k]<\infty$ for arbitrary $k\geq 1$. 

In some sense for the subcritical Galton-Watson process in the random environment, the population is vanishing "very fast". 

\subsection{Critical Galton-Watson process in random environment}
The most interesting  are the critical processes in the random environment. We call  Galton-Watson process in random environment critical if $\gamma=\mathcal{E}(\ln \displaystyle\frac{\beta}{p})=0 $ and the second moment $\mathcal{E}(\ln^2 \displaystyle\frac{\beta}{p})=Var(\ln \displaystyle\frac{\beta}{p})=\sigma^2<\infty$. We will return to the case $\sigma^2=\infty$ in the next paper.

We will consider two special cases for $\sigma^2$.

\begin{enumerate}
\item Strong critical case: $\sigma^2=0$

That is $\beta(i)=p(i)$, $i=1,2,3,\cdots$  and $S_t=\displaystyle 1+\frac{1-\beta_1}{\beta_1}+\frac{1-\beta_2}{\beta_2}+\cdots \frac{1-\beta_t}{\beta_t}$

If $0<\gamma=\mathcal{E}\ln\frac{1-\beta}{\beta}<\infty$, then $S_t \sim \gamma t$ $P$-a.s.

and $P(N_t\geq 1)= \displaystyle\frac{1}{S_t}\sim\frac{1}{\gamma t}$. 

Like in the homogeneous case, $EN_t=1$, $P(N_t\geq 1)\sim \displaystyle\frac{1}{\gamma t}$. That is, population slowly degenerates.

\item Much more important case is the case when $\sigma^2>0$. 

Here $EN_t=\prod\limits_{i=1}^t\displaystyle\frac{\beta_i}{p_i}=e^{\sum\limits_{i=1}^t \ln \displaystyle\frac{\beta}{p}(i)}=e^{\xi(t)}$, where $\xi(t)=\sum\limits_{i=1}^t \ln \displaystyle\frac{\beta}{p}(i)$. Due to the Law  of the Iterated Logarithm, 

\begin{equation}
\limsup \frac{\xi(t)}{\sigma \sqrt{t\ln\ln t}}=1
\end{equation}

\begin{equation}
\liminf \frac{\xi(t)}{\sigma \sqrt{t\ln\ln t}}=-1
\end{equation}
\end{enumerate}

That is, the sum $\xi(t)$ in the exponent has strong oscillations.

\begin{theorem}
If $\mathcal{E}(ln  \displaystyle\frac{\beta}{p})=0$  and $0<\mathcal{E}(\ln ^2\displaystyle\frac{\beta}{p})=\sigma^2<\infty$,  $E[N_t]$ will grow in an oscillating way. 
\end{theorem}

\begin{proof}

If $\mathcal{E}(ln  \displaystyle\frac{\beta}{p})=0$, and $\mathcal{E}(ln  \displaystyle\frac{\beta}{p})^2> <\infty$, by Lindeberg–Lévy Central Limit Theorem, $\frac{1}{\sqrt{t}}\sum\limits_{i=1}^t\ln( \displaystyle\frac{\beta_i}{p_i})\rightarrow N(0,1)$ in distribution. Thus $E[N_t]\sim e^{\sqrt{t}w}$ as $t\rightarrow \infty$ where $w$ is a standard random normal variable. This implies the expected number of particles at moment t will oscillates as $t$ goes to infinity.
\end{proof}

Consider now the series $S_t=\displaystyle 1+\frac{1-\beta_1}{\beta_1}+\frac{1-\beta_2}{\beta_2}e^{\ln \frac{p_1}{\beta_1}}+\cdots \frac{1-\beta_t}{\beta_t}\prod\limits_{j=1}^{t-1}e^{\ln \frac{p}{\beta}(j)}$

If the random variables $\beta_t,t=1,2,3\cdots$ are separated from 1 in some sense, say $\beta_t\leq 1-\delta_1$ for any $t\geq 0$, then from the previous formula, it follows that $S_t \rightarrow \infty$, $t\rightarrow \infty$ $P$-a.s. But it is true without additional assumptions. 

\begin{theorem}\label{Sinfcrit}
If $\gamma=\mathcal{E}\ln\frac{\beta}{p}=0$, $\sigma^2=\mathcal{E}\ln^2\frac{\beta}{p}<\infty$, then $S_{\infty}=\infty$, that is $P(N_{\infty}=0)=1$, in other words, $N_t=0$, for $t\geq \tau(\omega)<\infty$
\end{theorem}

\begin{proof}
Consider the sequences $A_n=2^{n^2}$, $B_n=2^n$ and $C_n=\sum\limits_{k=1}^n(A_k+B_k)$, then $C_n \sim 2^{n^2}$, $n\rightarrow \infty$. 

For $\xi(t)=\sum\limits_{j=1}^t\ln\displaystyle\frac{\beta}{p}(j)$, by Kolmogorov's inequality 

\begin{equation}
P(\max\limits_{k\leq C_n} \xi(k)\geq \sqrt{C_n}D_n)\leq \frac{\sigma^2}{D^2_n}
\end{equation}

and for $D_n=n$, due to the Borel Cantelli Lemma, we will get that 

\begin{equation}
\max\limits_{k\leq C_n} \xi(k) \leq \sqrt{C_n}n=2^{n^2/2}n\,\, \text{for}\,\, n\geq n_0(\boldsymbol{\omega_m)}
\end{equation}

And at the same time
\begin{equation}\label{50chance}
P(\xi(C_n+A_{n+1})-\xi(C_n)>\sqrt{A_{n+1}})=\frac{1}{2}
\end{equation}

and $A_{n+1}=2^{(n+1)^2}>>C_n$. As a result, there is about 50\% of $n$ for which we have the inequality in (\ref{50chance}). Here and only here, we used the assumption $A_n\sim c^n\sim 2^{n^2}$ (instead of $n^2$, one can use $n^{1+\epsilon}$, $\epsilon>0$.)

Finally, by Kolmogorov's inequality,

\begin{equation}
P(\max\limits_{k<B_{n(t)}}|\sum\limits_{j=C_n+A_{n+1}}^{C_n+A_{n+1}+k} \ln \frac{\beta}{p}(j)|>B_n^{\frac{1}{2}(1+\epsilon)} \leq \frac{C}{2^{n\epsilon}}
\end{equation}

Now we have to estimate the random variables $\frac{1-\beta_t}{\beta_t}$, $t=1,2,3,\cdots$ since very small values of these random variables can compensate the large values of $e^{-\xi(t)}=e^{\sum\limits_{j=1}^t \ln\frac{p}{\beta}}$. 

For some $\delta_1,\delta_2>0$, 

\begin{equation}
P(\frac{1-\beta_t}{\beta_t}\geq \delta_1)=P(\beta_t\leq \frac{1}{1+\delta_1})\geq \delta_2
\end{equation}

Consider the sequence $\frac{1-\beta_t}{\beta_t}, t\geq 1$ and divide it into series $\eta$ which collects all terms of $\frac{1-\beta_t}{\beta_t}<\delta_1$ and complementary series $\theta$ which collects all the terms $\frac{1-\beta_t}{\beta_t}\geq \delta_1$. The distribution of the length of such series is the geometric distribution and due to the Borel Cantelli Lemma $\max\limits_{i\leq N}\theta_i\sim c\ln N$ $P$-a.s.

The number of $\theta$-series between $C_n+A_{n+1}$ and $C_{n+1}$ again due to Borel-Cantelli lemma is growing exponentially in $n$ $P$-a.s., that is, inside the interval $(C_n+A_{n+1},C_{n+1})$, there are many random variables $\frac{1-\beta_t}{\beta_t}\geq \delta_1$. But in this interval,

\begin{equation}
|\eta(C_n+A_{n+1}+k)|\geq 2^{n^2/2}-2^{-(n+1)^{(1+\epsilon)/2}}
\end{equation}

if only $\displaystyle\eta(C_n+A_{n+1})-\eta(C_n)>2^{(n+1)^2/2}$. This observation complete the proof of Theorem (\ref{Sinfcrit})
\end{proof}
Let us now estimate $S_t$. Note that the main contribution to $S_t$ can give to the factors $\prod\limits_{i=1}^s\displaystyle\frac{p}{\beta}(i)=e^{\sqrt{s}\gamma+o(s)}$ but prefactors $\displaystyle\frac{p}{\beta}(s)$ can be very small if only $\beta$ can be very large,. Under appropriate control of the large values of $\displaystyle\frac{1-\beta}{\beta}$, we can find the asymptotic formula. 

\begin{theorem}
If for any $\delta_0\geq 0$, $\mathcal{E}|\ln\frac{1-\beta}{\beta}|^{2+\delta}=\mathcal{E}|\ln\frac{\beta}{1-\beta}|^{2+\delta}=c_0<\infty$, then

\begin{equation}
\frac{1}{\sigma\sqrt{t}}\ln S_t \xrightarrow[law]{} M_1
\end{equation}

where $M_1=\max \limits_{s\in [0,1]} W_s$ and $W_s$ is the standard 1D Brownian motion. 

\end{theorem}
\begin{proof}

By the Chebyshev inequality, for any $\epsilon >0$, 
\begin{equation}
P(|\ln \frac{1-\beta_t}{\beta_t}|\geq \epsilon \sqrt{t}) \leq \frac{\mathcal{E}(\ln\frac{1-\beta_t}{\beta_t})^{2+\delta}}{\epsilon^{2+\delta}t^{1+\delta/2}}
\end{equation}

By Borel Cantelli Lemma, for any $\epsilon>0 $, there exists $t_0>0$ such that when $t\geq t_0$, $\displaystyle\frac{1}{\sqrt{t}}|\ln \frac{1-\beta_t}{\beta_t}|<\epsilon$, that is 
\begin{align}\label{first sum bound}
e^{-\epsilon t^{\frac{1}{•2}}}<\frac{1-\beta_t}{\beta_t}<e^{\epsilon t^{\frac{1}{2}}}
\end{align}

Put $M_t=\max\limits_{s\leq t}\frac{\sum\limits_{i=1}^s \ln\frac{p}{\beta}(i)}{\sigma\sqrt{t}}$.

Then for large $t$, 

\begin{equation}
e^{-\epsilon \sqrt{t}}e^{\sigma\sqrt{t}M_t}\leq S_t\leq e^{\sigma \sqrt{t}}te^{\delta \sqrt{t}}
\end{equation}

That is 

\begin{equation}
M_t-\epsilon \leq \frac{\ln S_t}{\sigma \sqrt{t}}   \leq  \frac{\ln t}{\sigma \sqrt{t}}+\epsilon+M_t
\end{equation}

Due to the functional central limit theorem by Donsker Prokhorov, as $t\rightarrow \infty$,

\begin{equation}
\frac{1}{\sigma\sqrt{t}}\ln S_t \rightarrow M_t \xrightarrow[law]{} M_1
\end{equation}

where $M_1=\max \limits_{s\in [0,1]} W_s$ and $W_s$ is the standard 1D Brownian motion. 

\end{proof}

\section{Analysis of the Geometric Progression}
In the last section of the paper, we will give two examples when the distribution of the random geometric progression can be calculated explicitly. 

Let $S_{\infty}=\displaystyle\frac{1-\beta_1}{\beta_1}+\frac{1-\beta_2}{\beta_2}\frac{p_1}{\beta_1}+\cdots$ and $X_t=\displaystyle\frac{1-\beta_t}{\beta_t}$ and $\theta_t=\displaystyle\frac{p_t}{\beta_t}$. Assume that 
$X_t \in (0,1)$ and $\theta_t\in(0,1)$ are independent for fixed $t$ and the pairs $(X_t,\theta_t)$ are independent for different $t=2,3,4,\cdots$. 

Then
\begin{equation}\label{S_inf}
S_{\infty}=X_1+\theta_1S_{\infty}'
\end{equation}

where in the right part all three random variables are independent and $S_{\infty}\stackrel{law}{=}S_{\infty}'$. The random geometric series (\ref{S_inf}) appears in many applications, see for example \cite{cell}. 

We will present now two typical examples. For the general discussion on the properties of the sum in (\ref{S_inf}) about infinite divisibility, asymptotes etc., we will present in a separate paper.

\subsection{Example}
Let $X_t, t=1,2,3,\cdots$ have Erlang law with the index $r=1,2,\cdots$, i.e. the Laplace transform of $X_t$ is

\begin{equation}
Ee^{-\lambda X_t}=\frac{1}{(1+\lambda)^r}=\Phi(\lambda)
\end{equation}

And $\theta_t, t=1,2,3,\cdots$ have the degenerated $\beta$-density:

\begin{equation}
\pi(x)=\beta x^{\beta-1}\mathbf{1}_{[0,1]}(x),\beta>0
\end{equation} 

Then $\psi(\lambda)=Ee^{-\lambda S_{\infty}}=\Phi(\lambda)\displaystyle\int_0^1\psi(\lambda \xi)\pi(\xi)d \xi=\frac{1}{\lambda^{\beta}}\Phi(\lambda)\int_0^{\lambda} \psi(y) \beta y^{\beta-1}dy$

Thus 

\begin{equation}
\frac{\psi(\lambda)\lambda^{\beta}}{\Phi(\lambda)}=\beta\int_0^{\lambda}\psi(y)y^{\beta-1}dy
\end{equation}

\begin{equation}
(\frac{\psi}{\Phi})'\lambda^{\beta}+\frac{\psi}{\Phi}\beta\lambda^{\beta-1}=\beta\psi(\lambda)\lambda^{\beta-1}
\end{equation}

\begin{equation}
\psi(\lambda)=\Phi(\lambda)\displaystyle e^{-\beta\int_0^{\lambda}\frac{1-\Phi(s)}{s}ds}
\end{equation}

Or 
\begin{align*}
\psi(\lambda)&=\frac{1}{(1+\lambda)^r}\displaystyle e^{-\beta\int_0^{\lambda}\frac{(1+z)^r-1}{z(1+z)^r}dz}\\
&=\frac{1}{(1+\lambda)^r}\displaystyle e^{-\beta\int_0^{\lambda}\frac{1+(1+z)+\cdots+(1+z)^{r-1}}{(1+z)^r}dz}\\
&=\frac{1}{(1+\lambda)^{r}}\displaystyle e^{-\beta(\ln(1+\lambda)+(\frac{1}{1+\lambda}-1)+\frac{1}{2}(\frac{1}{(1+\lambda)^2}-1)+\cdots+\frac{1}{r-1}(\frac{1}{(1+\lambda)^{r-1}}-1))}\\
&=\frac{1}{(1+\lambda)^{r+\alpha}}\displaystyle e^{-\beta(\frac{1}{1+\lambda}-1)+\frac{1}{2}(\frac{1}{(1+\lambda)^2}-1)+\cdots+\frac{1}{r-1}(\frac{1}{(1+\lambda)^{r-1}}-1))}
\end{align*}

The last Laplace transform means that 
\begin{equation}
S_{\infty}=Y_1+(Z_{11}+\cdots Z_{1v_1})+(Z_{21}+\cdots Z_{2v_2})+\cdots+(Z_{r-1,1}+\cdots+Z_{r-1,v_{r-1}})
\end{equation}
where $Y_1$ has Gamma distribution with parameter $r+\alpha$, and $v_1,v_2,\cdots,v_{r-1}$ has Poissonian laws with parameter $\beta,\beta/2,\cdots,\beta/(r-1)$ respectively. And $Z_{l,j},l=1,2,\cdots, r-1, j=1,2,3,\cdots, r-1$ have Erlang laws of the rank $l$ and the number of the terms in each sum is $v_{l}$. 

In particular, if $r=1$, $\psi(\lambda)=\displaystyle\frac{1}{1+\lambda}e^{-\beta\int_0^{\lambda}\frac{1-\frac{1}{1+s}}{s}ds}=\frac{1}{(1+\lambda)^{1+\beta}}$,
thus $S_{\infty}$ has Gamma distribution with parameter $1+\beta$

The density of random variable $S_{\infty}$ has power asymptotics near 0. Note that the random variable $\theta_t,t=1,2,3,\cdots$ have positive density on $(0,1)$. The answer will be different if $\theta_t$ are separated from 0 and 1. Assume $\theta_t\equiv \rho \in (0,1)$ and consider the random geometric progression $S=\sum\limits_{n=0}^{\infty}X_n\rho^n$ where $0<\rho<1$ is a fixed number and $X_n,n=1,2,3\cdots$ have i.i.d Gamma distribution with parameter $\beta>0$.

Then  $E^{-\lambda X_i}=\displaystyle\int_0^{\infty}e^{-\lambda x}\frac{x^{\beta-1}e^{-x}}{\Gamma(\beta)}dx=\frac{1}{(1+\lambda)^{\beta}}$

We want to estimate $P(S<\delta)$, $\delta<<1$ and $P(S>A)$, $A>>1$. 

Note that 

\begin{equation}
Ee^{-\lambda S}=\frac{1}{(1+\lambda)^{\beta}(1+\rho\lambda)^{\beta}\cdots (1+\lambda\rho^n)}=\psi(\lambda)
\end{equation}

The function $\psi(\lambda)$ is analytic for $\lambda>-1$ and using this fact or the exponential Chebyshev's equality, one can find that $P(S>A)\sim cA^{\beta-1}e^{-\beta}$. Much more interesting problem is the estimation of $P(S<\delta)$. 

We have $P(S<\delta)=P(e^{-\lambda S}>e^{-\lambda \delta})\leq \frac{Ee^{-\lambda S}}{e^{-\lambda \delta}}=\psi(\lambda)e^{\lambda \delta}$

But $\psi(\lambda)=\prod\limits_{n=1}^{\infty}(1+\rho^n\lambda)^{-\beta}=\prod\limits_{n=1}^{n_0}(1+\rho^n\lambda)^{-\beta}\prod\limits_{n=n_0+1}^{\infty}(1+\rho^n\lambda)^{-\beta}$.

If $\rho^{n_0}\lambda \simeq 1$, then $\ln \lambda+n_0\ln \rho =0$, then $n_0\sim \displaystyle\frac{\ln \lambda}{\ln \frac{1}{\rho}}$

As easy to see, for $\lambda>>1$, 

$\psi(\lambda)=\prod\limits_{n=1}^{n_0}(1+\rho^n\lambda)^{-\beta}\prod\limits_{n=n_0+1}^{\infty}(1+\rho^n\lambda)^{-\beta}\sim C\lambda^{n_0}\rho^{n_0(n_0-1)/2}\sim C\lambda^{\frac{\ln \lambda}{\ln \frac{1}{\rho}}}\rho^{\frac{\ln^2\lambda}{2\ln^2\frac{1}{\rho}}}$.

Then (with log accuracy),

\begin{equation}
P(S<\delta)\leq e^{-\frac{\ln^2\lambda}{\ln \frac{1}{\rho}}+\frac{\ln^2\lambda}{2\ln\frac{1}{\rho}}+\delta \lambda}\sim e^{-\frac{\ln^2\lambda}{2\ln\frac{1}{\rho}}+\delta \lambda}
\end{equation}

Then when $\lambda=\displaystyle\frac{2\ln\lambda}{\delta\ln\frac{1}{\rho}}$,

\begin{equation}
P(S<\delta)=e^{(\ln\frac{1}{\delta}-\ln\ln\frac{1}{\delta}+C)^2+{2\ln\frac{1}{\rho}}+\frac{2\ln\frac{1}{\delta}}{\ln\frac{1}{\rho}}} \sim e^{-\frac{\ln^2\frac{1}{\delta}}{2\ln\frac{1}{\rho}}}
\end{equation}

Calculation of exact asymptotics is more difficult problem. For additional details, see (\cite{cell}).\\

\appendix


\acks
Dan Han was supported by University of Louisville EVPRI Grant "Spatial Population Dynamics with Disease" and AMS Research Communities "Survival Dynamics for Contact Process with Quarantine".
S. Molchanov was supported by the Russian Science Foundation RSF grant project 17-11-01098 and project 20-11-20119. 

%
%
%
%

\bibliographystyle{apt.bst}
\bibliography{mybib}

\end{document}